\documentclass[12pt]{amsart}

\usepackage {amsmath, amsopn}
\usepackage[applemac]{inputenc}
\usepackage {amsfonts, amssymb}
\usepackage {latexsym}

\newtheorem {theorem} {Theorem} [section]
\newtheorem {lemma} {Lemma} [section]
\newtheorem {prop} {Proposition} [section]

\newtheorem{preremark}{Remark}[section]
\newenvironment{remark}%
  {\begin{preremark}\rm}{\end{preremark}}
   \newtheorem{preremark1}{Example}[section]

     \newtheorem{preremark2}{Definition}[section]
\newenvironment{defn}
  {\begin{preremark2}\rm}{\end{preremark2}}

\DeclareMathOperator{\divo}{div}
\DeclareMathOperator{\Divo}{Div}

\begin{document}

\renewcommand{\theequation}{\arabic{section}.\arabic{equation}}

\numberwithin{equation}{section}

\title[Maxwell-alpha fluid]{Global generalized solutions for Maxwell-alpha and Euler-alpha equations}

\author[D.Vorotnikov]{Dmitry Vorotnikov}

\dedicatory{CMUC, Department of Mathematics, University of Coimbra\\ 3001-454 Coimbra, Portugal\\ mitvorot@mat.uc.pt}

\thanks{The research was partially supported by CMUC/FCT}




\begin{abstract} We study initial-boundary value problems for the Lagrangian averaged alpha models for the equations of motion for the corotational Maxwell and inviscid fluids in 2D and 3D. We show existence of (global in time) dissipative solutions to these problems. We also discuss the idea of dissipative solution in an abstract Hilbert space framework. 
\end{abstract}

\maketitle

\section {Introduction}

\newcommand{\sgn}{\text{sgn}}
\newcommand {\R} {\mathbb{R}}
\newcommand {\E} {\mathbf{E}}
\def\be{\begin{equation}}
\def\ee{\end{equation}}
\def\fr#1#2{\frac{\partial #1}{\partial #2}}

Recently, so-called alpha-models of fluid mechanics have attracted attention of many researchers, see e.g. \cite{cpam,busu,many, fc, fht, hkmr, kapr,hmr2,hou,shiz,mrs,sk} and references therein. These models turned out to be rather ubiquitous, being relevant in such issues as turbulence and large eddy simulations, and, on the other hand, being related to the second grade fluids. 

In this work, we study the initial boundary value problem for the corotational Maxwell-alpha viscoelastic fluid flow:
\begin{equation} \label{eq1}\frac {\partial v} {\partial t} + \sum\limits _ {i=1} ^ {n} u_i\frac {\partial
v} {\partial x_i} + \sum\limits _ {i=1} ^ {n}
 v_i\nabla u_i +\nabla p =\Divo \sigma, \end{equation}
\begin{equation}\label{eq2}\sigma + \lambda \left(\frac {\partial \sigma} {\partial t} + \sum\limits _ {i=1}
^ {n} u_i\frac {\partial \sigma} {\partial x_i}+ \sigma W - W\sigma\right) = 2\eta \mathcal
{E}, \end{equation}
\be \label{eq3}v=u-\alpha^2 \Delta u,\ee
\begin{equation} \label{eq4}\divo u =
0,\end{equation}
\begin{equation} \label{eq5}u \Big | _ {\partial\Omega} =0,\end{equation}
\begin{equation} \label{eq6}u | _ {t=0} =a, \ \sigma | _ {t=0} = \sigma_0.\end{equation}

Here, $ \Omega $ is a domain in the space $ \mathbb {R}
^n $, $n=2,3 $, $u $ is an unknown velocity vector, $p $ is an unknown
modified pressure, $ \sigma $ is an unknown deviatoric stress tensor, $v$ is an auxiliary variable (all of them
depend on points $x $ in the domain $ \Omega$, and on 
time $t $); $$ \mathcal {E} = \mathcal {E} (u) = 
(\mathcal {E} _ {ij} (u)),\ \mathcal {E} _ {ij} (u) = \frac {1}
{2} \left(\frac {\partial u_i} {\partial x_j} + \frac {\partial u_j}
{\partial x_i}\right),$$ is the strain velocity tensor, $$W = W (u)=
(W _ {ij} (u)),\ W _ {ij}  (u) = \frac {1}
{2} \left(\frac {\partial u_i} {\partial x_j} - \frac {\partial u_j}
{\partial x_i}\right),$$ is the vorticity tensor, $ \eta> 0 $ is the Maxwellian
viscosity of the fluid, $ \lambda>0 $ is the relaxation time, $\alpha>0$ is a scalar parameter, $a$
and $\sigma_0$ are given functions. The external force is, for simplicity, assumed to be zero.

Equation \eqref{eq1} is the Lagrangian averaged (Euler-alpha-like) Cauchy's equation of motion, \eqref{eq2}
is the corotational Maxwell constitutive law \cite{lar}, \eqref{eq4} is the
equation of continuity, and \eqref{eq5} is the no-slip
condition. 
When $\alpha=0$, \eqref{eq1} -- \eqref{eq6} becomes the initial-boundary value problem for the equations of motion for the corotational Maxwell viscoelastic fluid \cite{gs2,lar,rheo,book}.  Thus, \eqref{eq1} -- \eqref{eq6} can be considered as an appropriate $\alpha$-model for the corotational Maxwell fluid. 

In the particular case $\eta=0$ and $\sigma_0=0$, we recover the Dirichlet problem for the celebrated Euler-$\alpha$ model: 

\begin{equation} \label{eq1e}\frac {\partial v} {\partial t} + \sum\limits _ {i=1} ^ {n} u_i\frac {\partial
v} {\partial x_i} + \sum\limits _ {i=1} ^ {n}
 v_i\nabla u_i +\nabla p =0, \end{equation}
\be \label{eq3e}v=u-\alpha^2 \Delta u,\ee
\begin{equation} \label{eq4e}\divo u =
0,\end{equation}
\begin{equation} \label{eq5e}u \Big | _ {\partial\Omega} =0,\end{equation}
\begin{equation} \label{eq6e}u | _ {t=0} =a.\end{equation}

The Euler-alpha (also known as Lagrangian averaged Euler and inviscid Camassa-Holm) equations were introduced and derived in \cite{kapr,hmr2}.  They are well-posed on small time intervals \cite{hkmr,sk,mrs}. However, neither strong nor weak general global solvability result has been known for the three-dimensional domains. A sort of a Beale-Kato-Majda criterion for these equations was proposed in \cite{hou}. Some reviews of mathematical results on the Euler-$\alpha$ model and many references may be found in \cite{cpam,busu,fc,shiz}.

The Maxwell model is one of the basic and classical models of a viscoelastic material. Its mechanical analogy is comprised of a spring and a dashpot connected in series \cite{rheo}. The multidimensional Maxwell models generate complicated systems of PDEs due the frame-indifference restrictions and consecutive involvement of objective derivatives \cite{lar,book}. The simplest objective derivative  is the corotational (Jaumann) one, in the incompressible case its use yields \eqref{eq1}--\eqref{eq4} with $\alpha=0$.

Very few mathematical results are known for the corotational Maxwell fluid equations (see \cite{gs2,ren,renb}). In particular, there is no global solvability theorem, even in 2D. Moreover, there are evidences of non-existence of smooth solutions \cite{renb,gs2}. 
Let us also mention paper \cite{numeri} with some numerical issues regarding the corotational Maxwell fluid.

In these circumstances, if we want the problem to be solvable globally, a possible way out is to consider a kind of a generalized solution different from the standard hydrodynamical weak solution framework. We are going to use the concept of \emph{dissipative solution} due to P.-L. Lions. It was suggested in \cite{blions} for the Euler equations of ideal fluid flow, which have only been proven to be globally weakly solvable on the torus (this is a very recent result \cite{wied1}). Some properties of dissipative solutions are discussed in \cite{tit,booksr}. Later, existence of such solutions was established for Boltzmann's  equation \cite{lio1} (see also \cite{vilg}), and for various models arising in magnetohydrodynamics \cite{anwu,sra}, diffusion in polymers \cite{diss1}, and image restoration \cite{v14}.  

The notion of dissipative solution plays a key role in the problem of transition from kinetic theory to hydrodynamics. In the Euler hydrodynamic limit, the renormalized solutions of Boltzmann's equation tend to dissipative solutions of Euler's one \cite{bren,booksr,vil}. Other issues concerning relevance of dissipative solutions for the Euler equations and relation of this concept to the weak and measure-valued solutions may be encountered in \cite{bls,lel}. 

In math literature, the expression ``dissipative solution" has various meanings. In particular, the notion that we use differs from the ones from \cite{bres,duc,phs,port1,portc}.

The objective of our paper is to introduce dissipative solutions for the corotational Maxwell-alpha problem \eqref{eq1} -- \eqref{eq6}, and to show their existence and basic properties. These solutions are always global in time. In the appendix to the paper, we present the skeleton of the idea of dissipative solution via considering it in an abstract Hilbert space setting. 

Throughout the paper, for definiteness, we assume $\eta>0$. However, the results remain valid in the Euler-alpha case, and the proofs are similar but simplified (see also Remark \ref{eac}).

Our approach is not directly applicable to the corotational Maxwell fluid ($\alpha=0$), but it allows us to construct some ``ultra-generalized" solutions for that model (see Remark \ref{maxw}).

The paper is organized in the following way. The next section contains preliminary material which is required for the formulation of the main result (Theorem \ref{mai}). The proof of the theorem is provided in the third section, which also contains some discussion of related open problems.  
The appendix with a general explanatory approach to dissipative solutions is rather an illustration than a collection of results, therefore its framework is not general enough to encompass the Maxwell-alpha or Euler equations. 

\section {Preliminaries, notation and the main result}

By $ \mathbb {R} ^ {n\times n} $, we denote the space of $n\times n $-matrices with the following scalar product: for $A =
(A _ {ij}) $, $B = (B _ {ij}) $,
$$ A:B = \sum\limits _ {i, j=1} ^ {n}
A _ {ij} B _ {ij}, $$ and let $ \mathbb {R} ^ {n\times n} _S $ be its
subspace of symmetric matrices.

Below in the paper, $\Omega$ is considered to be a domain (i.e. an open set in $\R^n,\	 n=2,3$) possessing the \emph{cone property}. We recall \cite{adams} that this means that each point $x\in\Omega$ is a vertex of a \emph{finite cone} $C_x$ contained in $\Omega$, and all these cones $C_x$ are congruent. A finite cone is a set of the form $$C_x=B_1\cap \{x+\xi (y-x)|y\in B_2, \xi >0 \}$$ where $B_1$ and $B_2 $ are open balls in $\R^n$, $B_1$ is centered at $x$, and $B_2$ does not contain $x$. We require this property to have Sobolev embeddings. 

We use the standard notations $L_p (\Omega, F) $, $W_p^{\beta}
(\Omega, F) $, $H^{\beta} (\Omega, F) =$ $W_2^{\beta} (\Omega,F) $, $H^{\beta}_0 (\Omega, F) = \stackrel
{\circ}{W}{}_2^{\beta} (\Omega, F) $ $(\beta > 0)$ for the Lebesgue
and Sobolev spaces of functions with values in a finite-dimensional space $F$. In this notation, sometimes we only keep the function space symbol and omit $\Omega$ and $F$, especially when $F=\mathbb {R} ^ {n\times n} _S$.

The Euclidean norm in finite-dimensional spaces $F$ is denoted as $ | \cdot | $. The symbol $ \| \cdot \| $ will stand for the Euclidean norm in $L_2(\Omega, F)$, and $ \| \cdot \| _ \beta$ will stand for the same thing in $H^{\beta} (\Omega, F)$. The corresponding scalar products are denoted by $(\cdot,\cdot)$ and $(\cdot,\cdot)_\beta$, respectively.

Let $\mathcal {V}$ be the set
of smooth, divergence-free, compactly supported  in $ \Omega $  functions with
values in $\R^n$. The symbols $H $ and $V$ 
denote the closures of $ \mathcal {V} $ in $L_2$ and $H^1$, respectively. We also use the spaces $V_i=V\cap H^i, \ i=2,3$, with the scalar product inherited from $H^i$. 

There exists a continuous Leray projection
 $$P:L_2(\Omega, \R^3)\to H.$$

Let us introduce the operator  $$\Delta_\alpha= I - \alpha^2 \Delta,$$
where $\alpha$ is the same as in \eqref{eq3}, and $I$ is the identity map.

In the space $V$, along with the scalar product $(\cdot,\cdot)_1$, we use another one
$$(u,v)_V=(u,v)+(\alpha\nabla u, \alpha\nabla v),$$
and the corresponding Euclidean norm $\|\cdot\|_V$. Let us mention the inequality
 $$\|u\|_1\leq \max\{1,1/\alpha\}\|u\|_V.$$
 
We recall the following abstract observation \cite{temam, book}. Assume that we have two Hilbert spaces, $X\subset Y,$ with
continuous embedding operator $i:X\to Y$, and $i(X)$ is dense in
$Y$. The adjoint operator $i^*:Y^*\to X^*$ is continuous and,
since $i(X)$ is dense in $Y$, one-to-one. Since $i$ is one-to-one,
$i^*(Y^*)$ is dense in $X^*$, and one may identify $Y^*$ with a
dense subspace of $X^*$. Due to the Riesz representation theorem,
one may also identify $Y$ with $Y^*$. We arrive at the chain of
inclusions:
\begin{equation} X\subset Y\equiv Y^*\subset X^*.\end{equation}
Both embeddings here are dense and continuous. Observe that in
this situation, for $f\in Y, u\in X$, their scalar product in $Y$
coincides with the value of the functional $f$ from $X^*$ on the
element $u\in X$:
\begin{equation} \label{ttt} (f,u)_Y=\langle f,u \rangle. \end{equation}
Such triples $(X,Y, X^*)$ are called Lions triples. 
 
We use the Lions triples $(V_3,V,V_3^*)$ and $(H^2,L_2,H^{-2}_N)$. In the first triple, the structure on $V$ is determined by the scalar product  $(\cdot,\cdot)_V$. In the second one, we write $H^{-2}_N$ for $(H^2)^*$ due to relation of this space to Neumann boundary value problems \cite{neum}.   

The symbols $C (\mathcal{J}; E) $, $C_w (\mathcal{J}; E) $, $L_2
(\mathcal{J}; E) $ etc. denote the spaces of continu\-ous, weakly
continuous, quadratically integrable etc. functions on an interval
$\mathcal{J}\subset \mathbb {R} $ with
values in a Banach space $E $. We recall that a function $u:
\mathcal{J} \rightarrow E$ is \textit{weakly continuous} if for
any linear continuous functional $g$ on $E$ the function  $g(
u(\cdot)): \mathcal{J}\to \mathbb{R}$ is continuous.

We require the following spaces
$$ W_1= W_1(\Omega, T) = \{\tau\in L_2 (0, T; V_3), \
\tau' \in L_2 (0, T;
 {V_3^{*}}) \},$$ $$\|\tau\|_{W_1}=\|\tau\|_{L_2 (0, T; V_3)}+\|\tau'\|_{L_2 (0, T;
 {V_3^{*}}) },$$
 $$ W_2= W_2(\Omega, T) $$ $$= \{\tau\in L_2 (0, T; H^2(\Omega, \mathbb {R} ^ {n\times n} _S)), \
\tau' \in L_2 (0, T;
 H ^ {-2}_N(\Omega,\mathbb {R} ^ {n\times n} _S)) \},$$ $$\|\tau\|_{W_2}=\|\tau\|_{L_2 (0, T; H^2)}+\|\tau'\|_{L_2 (0, T; H ^ {-2}_N)}.$$
 Here and below the prime symbol stands for the time derivative. 
 Note that $$W_1\subset C([0,T];V), \ W_2\subset C([0,T];L_2),$$ 
\be\label{d21} \frac d {dt}(u,u)_V= 2\left\langle u',u\right\rangle,\ u \in W_1,\ee
\be\label{d22}\frac d {dt}(u,u)= 2\left\langle u',u\right\rangle,\ u \in W_2,\ee
and
\be\label{td1} \frac d {dt}(u,\phi)_V= \left\langle u',\phi\right\rangle,\ u \in W_1, \phi\in V_3.\ee
\be\label{td2}\frac d {dt}(u,\phi)= \left\langle u',\phi\right\rangle,\ u \in W_2, \phi\in H^2,\ee
these facts are consequences of e.g. \cite[Lemmas 2.2.7 and 2.2.8]{book} and formula \eqref{ttt} above.

Let us introduce the operators 
 $$A_3:V_3\to V^{*}_3,\  \langle A_3 u , \varphi \rangle= (u,\varphi)_3,$$ and
$$A_2:H^2\to H^{-2}_N,\  \langle A_2 \sigma, \Phi  \rangle= (\sigma,\Phi)_2,$$ where $\varphi$ and $\Phi$ are arbitrary elements of $V_3$ and $H^2$, respectively. 

Let $\mu=\eta/\lambda$.  Consider the following formal expressions, where $w$ and $\tau$ are vector- and matrix-valued functions of time, respectively, and $\delta$ is a positive number: 
 \be E_1(w,\tau)$$ $$=-\frac {\partial \Delta_\alpha w} {\partial t} - P\sum\limits _ {i=1} ^ {n} w_i\frac {\partial
\Delta_\alpha w} {\partial x_i} - P\sum\limits _ {i=1} ^ {n}
 (\Delta_\alpha w)_i\nabla w_i +P\Divo \tau,\ee 
 \be E_2(w,\tau)$$ $$= -\frac \tau \lambda  - \frac {\partial \tau} {\partial t} - \sum\limits _ {i=1}
^ {n} w_i\frac {\partial \tau} {\partial x_i}- \tau W(w) +W(w)\tau    + 2\mu \mathcal
{E}(w). \ee
 \be E_1(w,\tau,\delta)$$ $$=-\frac {\partial \Delta_\alpha w} {\partial t} - \delta P\sum\limits _ {i=1} ^ {n} w_i\frac {\partial
\Delta_\alpha w} {\partial x_i} - \delta P\sum\limits _ {i=1} ^ {n}
 (\Delta_\alpha w)_i\nabla w_i +\delta P\Divo \tau,\ee 
 \be E_2(w,\tau,\delta)$$ $$= -\frac {\delta\tau} \lambda  - \frac {\partial \tau} {\partial t} - \delta\sum\limits _ {i=1}
^ {n} w_i\frac {\partial \tau} {\partial x_i}- \delta\tau W(w) +\delta W(w)\tau    + 2\delta\mu \mathcal
{E}(w). \ee

The symbol $C$ will stand for a generic positive constant that can take different values in different lines. We shall sometimes write $C_\Omega$ to specify that the constant depends on the domain $\Omega$ \emph{only}.

Let us recall the well-known inequalities
  \begin {equation} \label{inq1} \|uv \|\leq C_\Omega \| u \| _ 2 \| v \| ,\ u\in
H^2, v\in L_2,
\end {equation}
\begin {equation} \label{inq2} \|uv \| \leq C_\Omega \| u \| _ 1 \| v \| _ 1,\ u,v\in
H^1\end {equation}
(see e.g. \cite[Corollary 2.1.1]{book}).

The following Gronwall-like lemma will be useful.
\begin{lemma} \label{ineq} (\cite[Lemma 3.1]{diss1}) Let $f,\chi, L, M: [0,T]\to \R$ be scalar functions, $\chi, L,M \in L_1(0,T),$ and $f\in
W^1_1(0,T)$ (i.e. $f$ is absolutely continuous). If \be \chi(t)
\geq 0, L(t) \geq 0\ee  and \be f'(t)+ \chi(t)\leq L(t)f(t)+M(t)
\ee for almost all $t\in (0,T)$, then \be f(t)+ \int\limits_0^t
\chi(s)\, ds\leq $$ $$ \exp\left(\int\limits_0^t L(s) ds\right)\left[f(0)+
\int\limits_0^t \exp\left(\int\limits_s^0 L(\xi) d\xi\right) M(s)\, ds
\right]\ee for all $t\in [0,T]$.
\end{lemma}

Now we can give

\begin{defn} \label{maindef} Let $a\in V$, $\sigma_0\in L_2(\Omega,\mathbb {R} ^ {n\times n} _S)$. A pair of
functions $(u,\sigma)$ from the class \be u\in C_w([0,\infty); V),\ \sigma\in
C_w([0,\infty); L_2(\Omega,\mathbb {R} ^ {n\times n} _S) ),\ee is called a {\it dissipative} solution
to problem \eqref{eq1} -- \eqref{eq6} if, 
for all test 
functions
$\zeta\in C^1([0,\infty); V_3),$ $\theta \in C^1([0,\infty); H^2(\Omega,\mathbb {R} ^ {n\times n} _S))$ and all non-negative
moments of time $t$,
one has \be \label{ds} 2\mu\|u(t)-\zeta(t)\|^2_{V}+
\|\sigma(t)-\theta(t)\|^2 $$ $$\leq
\exp\left(\int\limits_0^t \Gamma(s) ds\right)\Big\{ 
2\mu\|a-\zeta(0)\|^2_{V}+\|\sigma_0-\theta(0)\|^2 $$ $$+
\int\limits_0^t \exp\left(\int\limits_s^0 \Gamma(\xi) d\xi\right)
\big[4\mu\big( E_{1}(\zeta,\theta)(s),u(s)-\zeta(s)\big)$$
$$+ 2 \big( E_{2}(\zeta,\theta)(s), \sigma(s)-\theta(s)\big)\big]\, ds \Big\}
\ee
 where   $$\Gamma(t) =\gamma \max\{1,1/\alpha^2\} $$ \be\label{gam}\times \left(\|\Delta_\alpha \zeta(t)\|_1+ \|\zeta(t)\|_1+\alpha^2\|\zeta(t)\|_3 + \frac {(1+\mu)\|\theta(t)\|_2} {\mu}\right),\ee and $\gamma$ is a certain constant depending only on the properties of the domain $\Omega$. \end{defn}
 
Observe that the dissipative solutions satisfy the initial condition \eqref{eq6}. Indeed, at $t=0$, inequality \eqref{ds} becomes \be \label{ds0} 2\mu\|u(0)-\zeta(0)\|^2_{V}+
\|\sigma(0)-\theta(0)\|^2\leq
2\mu\|a-\zeta(0)\|^2_{V}+\|\sigma_0-\theta(0)\|^2 
\ee

This easily yields \eqref{eq6} (see a similar reasoning below, in the proof of Proposition \ref{prop41}).

Moreover, these solutions obey the following ``dissipative" estimate
\be \label{din1} 2\mu\|u(t)\|^2_{V}+
\|\sigma(t)\|^2 \leq
2\mu\|u(0)\|^2_{V}+\|\sigma(0)\|^2,\ \forall t>0.
\ee
It follows from \eqref{eq6} and \eqref{ds} with identically zero $\zeta$ and $\theta$. 

\begin{remark} \label{eac} In the Euler-alpha case, the definition is much simpler:
a
function $u\in C_w([0,\infty); V)$ is called a dissipative solution
to problem \eqref{eq1e} -- \eqref{eq6e} if, 
for all test 
functions
$\zeta\in C^1([0,\infty); V_3),$ and all non-negative
moments of time $t$,
one has \be \label{dsee} \|u(t)-\zeta(t)\|^2_{V} \leq
\exp\left(\int\limits_0^t \Gamma(s) ds\right)\Big\{ 
\|a-\zeta(0)\|^2_{V} $$ $$+
2\int\limits_0^t \exp\left(\int\limits_s^0 \Gamma(\xi) d\xi\right)
\big( E_{1}(\zeta,0)(s),u(s)-\zeta(s)\big)\, ds \Big\}
\ee
 where $$\Gamma(t) =\gamma\max\{1,1/\alpha^2\}\left(\|\Delta_\alpha \zeta(t)\|_1+ \|\zeta(t)\|_1+\alpha^2\|\zeta(t)\|_3 \right),  $$ and $\gamma$ is a certain constant depending only on $\Omega$. 
Obviously, \eqref{din1} becomes
\be \label{din2} \|u(t)\|^2_{V}\leq
\|u(0)\|^2_{V},\ \forall t>0.
\ee

\end{remark}

Our main result provides existence of dissipative solutions and their relation with the strong ones:

 \begin{theorem} \label{mai} Let $\Omega$ be a bounded domain having the cone property. a) Given $a\in V$, $\sigma_0\in L_2$, there is a
dissipative solution to problem \eqref{eq1} -- \eqref{eq6}.

b) If, for some $a\in V$, $\sigma_0\in L_2$, there exist $T>0$
and a strong solution $(u_T,\sigma_T) \in C^1([0,T]; V_3) \times C^1([0,T]; H^2)$ to problem \eqref{eq1} -- \eqref{eq6}, then the restriction
of any dissipative solution (with the same initial data) to
$[0,T]$ coincides with $(u_T,\sigma_T)$.

c) Every strong solution $(u,\sigma)\in C^1([0,\infty); V_3) \times C^1([0,\infty); H^2)$ is a (unique) dissipative solution.

\end{theorem}

\section{A regularization and passage to the limit}

In order to prove Theorem \ref{mai} via approximation, we consider the following auxiliary problem:

\begin{equation} \label{au1} \frac {\partial v} {\partial t} + \delta\sum\limits _ {i=1} ^ {n} u_i\frac {\partial
v} {\partial x_i} + \delta\sum\limits _ {i=1} ^ {n}
 v_i\nabla u_i +\nabla p+\varepsilon A_3 u =\delta \Divo \sigma, \end{equation}
\begin{equation}  \label{au2} \delta\sigma + \lambda \left(\frac {\partial \sigma} {\partial t} + \delta\sum\limits _ {i=1}
^ {n} u_i\frac {\partial \sigma} {\partial x_i}+ \delta\sigma W - \delta W\sigma +\varepsilon A_2 \sigma\right)   = 2\delta\eta \mathcal
{E}, \end{equation}
\be \label{au3}  v=\Delta_\alpha u,\ee
\begin{equation} \divo u =
0,\end{equation}
\begin{equation} u \Big | _ {\partial\Omega} =0,\end{equation}
\begin{equation} \label{au6}  u | _ {t=0} =\delta a, \ \sigma | _ {t=0} =\delta \sigma_0.\end{equation} 

Here, $\varepsilon>0$ and $0\leq \delta \leq 1$ are parameters. The first one will then go to zero, and the second one is needed for correct application of Schaeffer's theorem \cite[p. 539]{evans}.
We keep assuming $a\in V,$ $\sigma_0\in L_2$.

The weak formulation of \eqref{au1} -- \eqref{au6} is as follows. 

\begin{defn} A pair of
functions $(u,\sigma)$ from the class \be u \in W_1,\ \sigma\in
W_2\ee is a {\it weak} solution to problem \eqref{au1} -- \eqref{au6} if the equalities

   $$\frac {d} {d t} (u, \varphi)_V -
   \delta\sum\limits^n _ {i=1} (u_i v, \frac {\partial\varphi} {\partial x_i}) + \delta\sum\limits _ {i=1} ^ {n}
   (v_i \nabla u_i, \varphi) $$ \begin{equation}  \label{w1}  +  \varepsilon (u,\varphi)_3 + \delta(\sigma, \nabla \varphi) = 0,
   \end{equation}
   and
    $$\frac {d} {d t} (\sigma, \Phi) + \frac {\delta} {\lambda} (\sigma, \Phi)-
   \delta\sum\limits^n _ {i=1} (u_i\sigma, \frac {\partial\Phi} {\partial x_i})   $$ \begin{equation}  \label{w2}  +\delta(\sigma
W-W\sigma, \Phi)  + \varepsilon (\sigma, 
\Phi)_2 =2\delta\mu
   (\nabla u, \Phi).\end{equation}   are satisfied  for all
 $ \varphi\in V_3, \ \Phi\in H^2(\Omega,\mathbb {R} ^ {n\times n} _S) $ almost everywhere in $(0, T)$, and \eqref{au3} and \eqref{au6} hold. \end{defn}
 
 \begin{remark} This notion of weak solution can be derived in a standard framework of multiplying by a test function, integrating by parts and using formula \eqref{td2} (see e.g. \cite{temam} concerning the Navier-Stokes system,  see also \cite[Section 6.1.1]{book} for general remarks about weak formulation of equations). Note that formula \eqref{td1} should not be  exploited at this stage. \end{remark}

\begin{lemma} \label{diin} Let $\Omega$ be any domain having the cone property.  Let $u,\sigma$ be a weak solution to problem \eqref{au1} -- \eqref{au6}. Then, for all 
$\zeta\in C^1([0,\infty); V_3),$ $\theta \in C^1([0,\infty); H^2)$ and $0\leq t\leq T$,
one has \be \label{disdel} 2\mu\|u(t)-\zeta(t)\|^2_{V}+
\|\sigma(t)-\theta(t)\|^2 $$ $$+ 2\varepsilon\int\limits_0^t
(2\mu\|u(s)-\zeta(s)\|_3^2+\|\sigma(s)-\theta(s)\|_2^2)\, ds$$ $$\leq
\exp\left(\int\limits_0^t \delta\Gamma(s) ds\right)\Big\{ 
2\mu\|\delta a-\zeta(0)\|^2_{V}+\|\delta\sigma_0-\theta(0)\|^2 $$ $$+
\int\limits_0^t \exp\left(\int\limits_s^0 \delta\Gamma(\xi) d\xi\right)
\big[4\mu\big( E_{1}(\zeta,\theta,\delta)(s),u(s)-\zeta(s)\big)$$
$$+ 2 \big( E_{2}(\zeta,\theta,\delta)(s), \sigma(s)-\theta(s)\big)-4\mu\varepsilon (\zeta(s),u(s)-\zeta(s))_3 $$
$$ -2\varepsilon(\theta(s), \sigma(s)-\theta(s))_2 \big]\, ds \Big\}
\ee where $\Gamma$ is as in \eqref{gam}. \end{lemma}


\begin{proof}  Observe that $$\frac {d} {d t} (\zeta, \varphi)_V -
   \delta\sum\limits^n _ {i=1} (\zeta_i \Delta_\alpha\zeta, \frac {\partial\varphi} {\partial x_i}) +  \delta\sum\limits _ {i=1} ^ {n}
   ((\Delta_\alpha\zeta)_i \nabla \zeta_i, \varphi) +(E_{1}(\zeta,\theta,\delta),\varphi)$$ \begin{equation} \label{z1} +  \varepsilon (\zeta,\varphi)_3 + \delta(\theta, \nabla \varphi) = \varepsilon (\zeta,\varphi)_3,
   \end{equation}
   and
    $$\frac {d} {d t} (\theta, \Phi) + \frac { \delta} {\lambda} (\theta, \Phi)- \delta
   \sum\limits^n _ {i=1} (\zeta_i\theta, \frac {\partial\Phi} {\partial x_i})   + \delta(\theta
W(\zeta)-W(\zeta)\theta, \Phi) $$ \begin{equation} \label{z2}+ (E_{2}(\zeta,\theta,\delta),\Phi) + \varepsilon (\theta, 
\Phi)_2 =2 \delta\mu
   (\nabla \zeta, \Phi)+ \varepsilon (\theta, 
\Phi)_2.\end{equation} 
      for
 $ \varphi\in V_3, \ \Phi\in H^2 $. 
Denote $w=u-\zeta$ and $\varsigma=\sigma-\theta$. For almost all $t\in(0,T)$, put $\varphi=w(t)$ and $\Phi=\varsigma(t)$. Multiply the difference between \eqref{w1} and \eqref{z1} by $2\mu$, and add this to the difference between \eqref{w2} and \eqref{z2}, arriving at

$$ \mu \frac {d} {d t} (w, w)_V +\frac 1 2 \frac {d} {d t} (\varsigma, \varsigma)$$ $$-
   2 \delta\mu\sum\limits^n _ {i=1} (\zeta_i \Delta_\alpha w, \frac {\partial w} {\partial x_i}) -
   2 \delta\mu\sum\limits^n _ {i=1} (w_i \Delta_\alpha \zeta, \frac {\partial w} {\partial x_i}) - 2 \delta\mu\sum\limits^n _ {i=1} (w_i \Delta_\alpha w, \frac {\partial w} {\partial x_i})$$ $$+  2 \delta\mu\sum\limits _ {i=1} ^ {n}
   ((\Delta_\alpha \zeta)_i \nabla w_i, w) +2 \delta\mu\sum\limits _ {i=1} ^ {n}
   ((\Delta_\alpha w)_i \nabla \zeta_i, w)+2 \delta\mu\sum\limits _ {i=1} ^ {n}
   ((\Delta_\alpha w)_i \nabla w_i, w)$$ $$ + \frac { \delta} {\lambda} (\varsigma, \varsigma)-
    \delta\sum\limits^n _ {i=1} (u_i\varsigma, \frac {\partial\varsigma} {\partial x_i})   -
    \delta\sum\limits^n _ {i=1} (w_i\theta, \frac {\partial\varsigma} {\partial x_i})  $$ $$+ \delta(\theta
W(w)-W(w)\theta, \varsigma)  + \delta(\varsigma
W(u)-W(u)\varsigma, \varsigma)  + \varepsilon (\varsigma, 
\varsigma)_2 +  2\mu\varepsilon (w,w)_3$$ \begin{equation}  =2\mu(E_{1}(\zeta,\theta, \delta),w)\label{zzzz} + (E_{2}(\zeta,\theta, \delta),\varsigma) -2\mu\varepsilon (\zeta, 
w)_3-\varepsilon (\theta, 
\varsigma)_2.\end{equation} 

We will need the following equalities:
\be\label{equal}-\sum\limits^n _ {i=1} (\kappa_i \Delta_\alpha  \kappa, \frac {\partial \kappa} {\partial x_i}) + \sum\limits _ {i=1} ^ {n}
   (\Delta_\alpha \kappa_i \nabla \kappa_i, \kappa)=0. \ee
\be \label{equal1} \sum\limits^n _ {i=1} (\kappa_i\tau, \frac {\partial\tau} {\partial x_i})=0. \ee
Here $\kappa \in V_3$ and $\tau\in H^2$.
The first equality is trivial. The second one is well-known and may be obtained via integration by parts.

Note that \be(\varsigma
W(u)-W(u)\varsigma, \varsigma)(t)=\sum\limits_{i,j,k=1}^n \int\limits_{\Omega}\big(\varsigma_{ij}(t,x)
W_{jk}(t,x)\varsigma_{ik}(t,x)$$ $$-W_{jk}(t,x)\varsigma_{ki}(t,x)
\varsigma_{ji}(t,x)\big)\,dx=0, \ee
since $\varsigma$ is a symmetric matrix. Moreover,
 \be - \sum\limits^n _ {i=1} (w_i \Delta_\alpha w, \frac {\partial w} {\partial x_i})+\sum\limits _ {i=1} ^ {n}
   ((\Delta_\alpha w)_i \nabla w_i, w)=0,\ee 
   and
   \be \sum\limits^n _ {i=1} (u_i\varsigma, \frac {\partial\varsigma} {\partial x_i})=0 \ee
   (by \eqref{equal} and \eqref{equal1}, respectively.). 

Now, let us estimate the remaining nonlinear terms in \eqref{zzzz}. Integrating by parts, we get

$$\sum\limits^n _ {i=1} (\zeta_i \Delta_\alpha w, \frac {\partial w} {\partial x_i})=\sum\limits^n _ {i=1} ( \zeta_i w, \frac {\partial w} {\partial x_i})+\alpha^2\sum\limits^n _ {i,j=1} ( \nabla\zeta_i \nabla w_j, \frac {\partial w_j} {\partial x_i})$$ $$+\alpha^2\sum\limits^n _ {i,j=1} (\zeta_i \nabla w_j, \frac {\partial \nabla w_j} {\partial x_i}).$$

The first and the last terms vanish by \eqref{equal1}. The second one, due to the Cauchy-Bunyakovsky-Schwarz inequality and \eqref{inq1}, does not exceed $C_\Omega \alpha^2\|\zeta\|_3\|w\|_1^2.$
By \eqref{inq2},
$$\sum\limits^n _ {i=1} (w_i \Delta_\alpha \zeta, \frac {\partial w} {\partial x_i}) \leq C_\Omega \|\Delta_\alpha \zeta\|_1\|w\|_1^2,$$
and 
$$\sum\limits _ {i=1} ^ {n}
   ((\Delta_\alpha \zeta)_i \nabla w_i, w) \leq C_\Omega \|\Delta_\alpha \zeta\|_1\|w\|_1^2.$$
Then,
   $$\sum\limits _ {i=1} ^ {n}
   ((\Delta_\alpha w)_i \nabla \zeta_i, w)= \sum\limits _ {i=1} ^ {n}
   (w_i \nabla \zeta_i, w)$$ $$+\alpha^2\sum\limits^n _ {i,j=1}(\nabla w_i \nabla \frac{\partial\zeta_i}{\partial x_j}, w_j)+\alpha^2\sum\limits^n _ {i,j=1}(\nabla w_i \frac{\partial\zeta_i}{\partial x_j}, \nabla w_j)$$
     $$\leq
   C_\Omega \|\zeta\|_1\|w\|^2_1+C_\Omega \alpha^2\|\zeta\|_3\|w\|^2_1.$$
  Further, by integration by parts, $$-\sum\limits^n _ {i=1} (w_i\theta, \frac {\partial\varsigma} {\partial x_i})=$$
$$\sum\limits^n _ {i=1} (\frac {\partial w_i} {\partial x_i} \theta, \varsigma)+\sum\limits^n _ {i=1} (w_i \frac {\partial\theta} {\partial x_i}, \varsigma).$$
The first term is zero since $w$ is divergence-free, and the second one, by \eqref{inq2}, is bounded by $C_\Omega \|w\|_1\|\theta\|_2 \|\varsigma\|.$  Finally, by \eqref{inq1}, $$(\theta
W(w)-W(w)\theta, \varsigma)\leq C_\Omega \|w\|_1\|\theta\|_2 \|\varsigma\|.$$

Now, for certain $\gamma$ depending on $\Omega$ only, \eqref{zzzz} yields

$$ \frac {d} {d t} (2\mu \|w\|_V^2 +\|\varsigma\|^2)+ 2\varepsilon(2\mu \|w\|_3^2 + \|\varsigma\|_2^2)$$ $$\leq 2\delta\mu C_\Omega \|\Delta_\alpha \zeta\|_1\|w\|_1^2+ 2\delta\mu C_\Omega \|\zeta\|_1\|w\|^2_1$$ $$+2\delta\mu C_\Omega \alpha^2\|\zeta\|_3\|w\|^2_1+\delta C_\Omega \|w\|_1\|\theta\|_2 \|\varsigma\|$$ $$  +4\mu(E_{1}(\zeta,\theta,\delta),w) + 2(E_{2}(\zeta,\theta,\delta),\varsigma) -4\mu\varepsilon (\zeta, 
w)_3-2\varepsilon (\theta, 
\varsigma)_2$$ $$\leq \delta\gamma \left(2\mu \max\{1,1/\alpha^2\} \|w\|_V^2 \big(\|\Delta_\alpha \zeta\|_1+ \|\zeta\|_1+\alpha^2\|\zeta\|_3+ \|\theta\|_2\big) +\|\varsigma\|^2 \frac {\|\theta\|_2} {\mu}\right)$$ $$ +4\mu(E_{1}(\zeta,\theta,\delta),w) + 2(E_{2}(\zeta,\theta,\delta),\varsigma) -4\mu\varepsilon (\zeta, 
w)_3-2\varepsilon (\theta, 
\varsigma)_2$$ $$\leq \delta\Gamma (2\mu \|w\|_V^2 +\|\varsigma\|^2)$$ \begin{equation}  +4\mu(E_{1}(\zeta,\theta,\delta),w)\label{z4} + 2(E_{2}(\zeta,\theta,\delta),\varsigma) -4\mu\varepsilon (\zeta, 
w)_3-2\varepsilon (\theta, 
\varsigma)_2,\end{equation}
with $\Gamma$ from \eqref{gam}.  It remains to apply Lemma \ref{ineq} to this inequality.
\end{proof}

\begin{lemma} \label{leae} Let $u,\sigma$ be a weak solution to problem \eqref{au1} -- \eqref{au6}. The following estimates are valid: \be \label{ae1} 2\mu
\|u\|_{L_\infty(0,T;V)}+\|\sigma\|_{L_\infty(0,T;L_2)}\leq 2\mu
\|a\|_V+\|\sigma_0\|=C,\ee \be \label{ae2}
\|u\|_{L_2(0,T;V_3)}+\|\sigma\|_{L_2(0,T;H^2)}\leq  \frac C{\sqrt{\varepsilon}},\ee
\be\label{ae3} \|u'\|_{L_2(0,T;V_3^*)}+\|\sigma'\|_{L_2(0,T;H^{-2}_N)}\leq  C (1+
\sqrt{\varepsilon}).\ee  The constants $C$ are
independent of $\varepsilon$ and $\delta$, but depend on $\|a\|_{V}$,
$\|\sigma_0\|$, $T$.
\end{lemma}

\begin{proof} The estimates \eqref{ae1} and \eqref{ae2} are direct consequences of \eqref{disdel} with $\zeta \equiv \theta\equiv 0$. It
remains to estimate the time derivatives, expressing them from \eqref{w1} and \eqref{w2}, and taking into account \eqref{td1} and  \eqref{td2}: \be\label{dert}\|\langle u', \varphi \rangle\|_{L_2(0,T)}\leq
  \delta\|\sum\limits^n _ {i=1} (u_i u, \frac {\partial\varphi} {\partial x_i})\|_{L_2(0,T)} + \delta\|\sum\limits _ {i=1} ^ {n}
   (u_i \nabla u_i, \varphi)\|_{L_2(0,T)}$$ 
   $$+\delta\alpha^2 \|\sum\limits^n _ {i=1} (u_i \Delta u, \frac {\partial\varphi} {\partial x_i})\|_{L_2(0,T)} + \delta\alpha^2\|\sum\limits _ {i=1} ^ {n}
   (\Delta u_i \nabla u_i, \varphi)\|_{L_2(0,T)}$$  $$+  \varepsilon \|(u,\varphi)_3\|_{L_2(0,T)} + \delta\|(\sigma, \nabla \varphi)\|_{L_2(0,T)}, \ee
   and
   \be\label{derts}\|\langle \sigma', \Phi \rangle\|_{L_2(0,T)}\leq  \frac {\delta} {\lambda} \|(\sigma, \Phi)\|_{L_2(0,T)}+
   \delta\|\sum\limits^n _ {i=1} (u_i\sigma, \frac {\partial\Phi} {\partial x_i})\|_{L_2(0,T)} $$  $$ +\delta\|(\sigma
W-W\sigma, \Phi)\|_{L_2(0,T)}  $$ $$+ \varepsilon\| (\sigma, 
\Phi)_2\|_{L_2(0,T)} +2\delta\mu
   \|(\nabla u, \Phi)\|_{L_2(0,T)}.\ee 

Integration by parts implies   \be\sum\limits^n _ {i=1}\left(u_i \Delta u, \frac {\partial\varphi} {\partial x_i}\right)$$ $$=-\sum\limits^n _ {i,j=1}\left(\frac {\partial u_i}{\partial x_j} \frac {\partial u} {\partial x_j}, \frac {\partial\varphi} {\partial x_i}\right)-\sum\limits^n _ {i,j=1}\left(u_i \frac {\partial u} {\partial x_j}, \frac {\partial^2\varphi} {\partial x_i \partial x_j}\right),\ee
and \be\sum\limits _ {i=1} ^ {n}
   (\Delta u_i \nabla u_i, \varphi)=$$ $$-\sum\limits _ {i,j,k=1} ^ {n}
   \left( \frac {\partial u_i} {\partial x_j}\frac{\partial u_i}{\partial x_k}, \frac {\partial \varphi_k}{\partial x_j}\right) -\sum\limits _ {i,j,k=1} ^ {n} \left(\varphi_k \frac {\partial u_i} {\partial x_j}, \frac{\partial^2 u_i}{\partial x_j \partial x_k}\right)\ee
   The second term is zero due to \eqref{equal1}. 

Using H\"{o}lder's and Sobolev imbedding inequalities, we estimate \eqref{dert} as follows: 
   
   $$\|\langle u', \varphi \rangle\|_{L_2(0,T)} \leq \delta\|u\|^2_{L_4(0,T;L_2)}\|\nabla\varphi\|_{L_\infty}$$ $$+ \delta\|u\|_{L_\infty(0,T;L_2)}\|u\|_{L_2(0,T;V)}\|\nabla\varphi\|_{L_\infty}+2\delta\alpha^2 \|u\|^2_{L_4(0,T;V)}\|\nabla\varphi\|_{L_\infty}$$ $$+\delta\alpha^2 \|u\|_{L_\infty(0,T;L_4)}\|u\|_{L_2(0,T;V)}\|\nabla^2\varphi\|_{L_4}$$ $$+\varepsilon \|u\|_{L_2(0,T;V_3)}\|\varphi\|_{3}+\delta\|\sigma\|_{L_2(0,T;L_2)}\|\nabla\varphi\|_{L_2}$$ $$ \leq C\left(\|u\|^2_{L_\infty(0,T;V)}\|\varphi\|_{3}+\varepsilon \|u\|_{L_2(0,T;V_3)}\|\varphi\|_{3}+\|\sigma\|_{L_\infty(0,T;L_2)}\|\varphi\|_{3}\right)$$
   Now, \eqref{ae1} and \eqref{ae2} yield  \be\|\langle u', \varphi \rangle\|_{L_2(0,T)} \leq C(1+\sqrt{\varepsilon})\|\varphi\|_{3},\ee which is equivalent to the required bound for $u'$. 
   
 Similarly, from \eqref{derts} we derive the estimate for the time derivative of $\sigma$: \be\|\langle \sigma', \Phi \rangle\|_{L_2(0,T)}\leq C[\|\sigma\|_{L_2(0,T;L_2)}\|\Phi\| $$ $$ +\|u\|_{L_\infty(0,T;L_4)}\|\sigma\|_{L_2(0,T;L_2)}\|\nabla\Phi\|_{L_4}+\|\sigma\|_{L_2(0,T;L_2)}\|u\|_{L_\infty(0,T;V)}\|\Phi\|_{L_\infty}$$ $$+\varepsilon \|\sigma\|_{L_2(0,T;H^2)}\|\Phi\|_{2}+ \|u\|_{L_2(0,T;V)}\|\Phi\|]$$ $$\leq C[\|\sigma\|_{L_\infty(0,T;L_2)}\|\Phi\|_2 +\|u\|_{L_\infty(0,T;V)}\|\sigma\|_{L_\infty(0,T;L_2)}\|\Phi\|_{2}$$ $$+\varepsilon \|\sigma\|_{L_2(0,T;H^2)}\|\Phi\|_{2}+ \|u\|_{L_\infty(0,T;V)}\|\Phi\|_2]$$ $$\leq  C(1+\sqrt{\varepsilon})\|\Phi\|_{2}.\ee
 \end{proof}

\begin{lemma} \label{lews} Given $T>0$, a bounded domain $\Omega$ having the cone property, and data $a\in V,$ $\sigma_0\in L_2,$ there exists a weak solution to problem \eqref{au1} -- \eqref{au6} with $\delta=1$. \end{lemma}

\begin{proof}  Observe that Lemma \ref{leae} yields the a priori estimate \be \label{aes}
\|u\|_{W_1}+ \|\sigma\|_{W_2}\leq C,\ee  where $C$ may depend on $\varepsilon$ but does not depend on $\delta$.

Let us rewrite the weak statement of \eqref{au1} -- \eqref{au6} in the suitable operator form \begin{equation} \label{opeq} \tilde {A} (u, \sigma) = \delta Q (u, \sigma). \end{equation} For this purpose, we introduce the following operators. Here $\varphi\in V_3$ and $\Phi\in H^2$ are test functions. 

$$N_1: W_1\to L_2 (0, T; V^{*}_3), $$ $$ \langle N_1 (u), \varphi \rangle =
 \sum\limits^n _ {i=1} (u_i \Delta_\alpha u, \frac {\partial\varphi} {\partial x_i}),$$
 
 $$N_2: W_1\to L_2 (0, T; V^{*}_3), $$ $$\langle N_2 (u), \varphi \rangle=-
 \sum\limits _ {i=1} ^ {n}
   (\Delta_\alpha u_i \nabla u_i, \varphi),$$
   
   $$N_3: W_1\times W_2\to L_2 (0, T; H^{-2}_N), $$ $$\langle N_3 (u,\sigma), \Phi \rangle =\sum\limits^n _ {i=1} (u_i\sigma, \frac {\partial\Phi} {\partial x_i}) ,$$

 $$N_4: W_1\times W_2\to L_2 (0, T; H^{-2}_N), $$ $$\langle N_4 (u,\sigma), \Phi \rangle =(W\sigma-\sigma
W, \Phi),$$

$$S_1: W_2\to L_2 (0, T; V^{*}_3),\ \langle S_1 \sigma, \varphi \rangle =
- (\sigma, \nabla \varphi) ,$$
 
$$S_2: W_1\to L_2 (0, T; H^{-2}_N),\ \langle S_2 u, \Phi \rangle =2\mu
   (\nabla u, \Phi),$$
   
$$S_3: W_2\to L_2 (0, T; H^{-2}_N),\ S_3 \sigma =-
\frac \sigma \lambda,$$

 $$Q:W_1\times W_2\to L_2 (0, T; V^{*}_3)\times L_2 (0, T; H^{-2}_N)
\times V\times L_2, $$

$$Q(u,\sigma)$$ $$=(N_1(u)+N_2(u)+S_1 \sigma, N_3(u,\sigma)+N_4(u,\sigma)+S_2 u +S_3 \sigma, a, \sigma_0),$$

$$\tilde{A}:W_1\times W_2\to L_2 (0, T; V^{*}_3)\times L_2 (0, T; H^{-2}_N)
\times V\times L_2,$$ $$
\tilde{A} (u,\sigma)= (u'+ \varepsilon A_3 u, \sigma' +\varepsilon A_2 \sigma, u | _ {t=0}, \sigma | _ {t=0}).
$$

We have treated the time derivative terms with the help of \eqref{td1} and \eqref{td2}.

Let us check that the operator $Q$ is compact (and, hence, continuous, since it consists of linear, bilinear and quadratic components). It suffices to show compactness of $S_1$, $S_2$, $S_3$, $N_1$, $N_2$, $N_3$ and $N_4$. 

Since the embeddings $V_3\subset V$ and $H^2\subset L_2$ are compact, the embeddings $W_1\subset L_2 (0, T; V), $ $W_2\subset
L_2 (0, T; L_2) $ are also compact by the Aubin-Simon theorem (see e.g. \cite{sim,book} or 
\cite[Theorem III.2.1]{temam}). Thus, $S_1$, $S_2$ and $S_3$ are compact as superpositions of compact embeddings with 
bounded operators. 

Moreover, by \cite[Corollary 8]{sim}, the embeddings $$W_1\subset L_p (0, T; V), W_2\subset
L_p (0, T; L_2),$$ and $$W_1\subset L_q (0, T; V_2)$$ are compact for any $p<\infty$ and $q<4$. 

It remains to establish boundedness (and, therefore, continuity) of the bilinear operators 

  $$N_3: L_4 (0, T; V) \times L_4 (0, T; L_2) \to L_2 (0, T; H^{-2}_N), $$

 $$N_4: L_4 (0, T; V) \times L_4 (0, T; L_2) \to L_2 (0, T; H^{-2}_N), $$ 

$$\tilde{N}_1: L_6 (0, T; V) \times L_3 (0, T; V_2) \to L_2 (0, T; V^{*}_3), $$ 

$$ \langle \tilde{N}_1 (u,w), \varphi \rangle =
 \sum\limits^n _ {i=1} (u_i \Delta_\alpha w, \frac {\partial\varphi} {\partial x_i}),$$

 $$\tilde{N}_2: L_3 (0, T; V_2) \times L_6 (0, T; V)  \to L_2 (0, T; V^{*}_3), $$ 
 
 $$\langle \tilde{N}_2 (u,w), \varphi \rangle=-
 \sum\limits _ {i=1} ^ {n}
   (\Delta_\alpha u_i \nabla w_i, \varphi).$$
   This is straightforward, e.g. for $\tilde{N}_2$ one has 
 
 $$\|\langle \tilde{N}_2 (u,w), \varphi \rangle\|_{L_2(0,T)}\leq C\|\Delta_\alpha u\|_{L_3 (0, T; L_2)} \|\nabla w\|_{L_6 (0, T; L_2)}\|\varphi\|_{L_\infty}$$ 
 $$\leq C\|u\|_{L_3 (0, T; V_2)} \|w\|_{L_6 (0, T; V)}\|\varphi\|_{3}. $$
 
 The linear operator $\tilde{A}$ is invertible by \cite[Lemma 3.1.3]{book}. Thus, \eqref{opeq} can be rewritten as \begin{equation} \label{opeq} (u, \sigma) = \delta \tilde{A}^{-1}Q (u, \sigma) \end{equation}
 in the space $W_1\times W_2$.
 
By virtue of Schaeffer's theorem \cite[p. 539]{evans},  the a priori estimate \eqref{aes}, uniform in $\delta$, guarantees  existence of a fixed point of the map $\tilde{A}^{-1}Q $, which is the required solution.  
\end{proof}

We are now in a position to prove the main result.

\begin{proof} (Theorem \ref{mai}) Take an increasing sequence of
positive numbers $T_m\to\infty $ and a decreasing sequence of
positive numbers $\varepsilon_m\to 0$. By Lemma \ref{lews}, there
is a pair $ (u_m, \sigma_m) $ which is a weak solution to problem
\eqref{au1} -- \eqref{au6} with $\delta=1$, $T=T_m$,
$\varepsilon=\varepsilon_m$.
 Denote by $ \widetilde {u} _m $
and $ \widetilde {\sigma} _m $ the functions which are equal to $u_m
$ and $ \sigma_m $ in $ [0, T_m] $ and are equal to zero on $ (T_m,
+\infty) $.

Lemma \ref{diin} implies that, for all $\zeta\in C^1([0,\infty); V_3),$ $\theta \in C^1([0,\infty); H^2)$ and $0\leq t\leq
T\leq T_m$, one has \be \label{in330} 2\mu\|u_m(t)-\zeta(t)\|^2_{V}+
\|\sigma_m(t)-\theta(t)\|^2 $$ $$\leq
\exp\left(\int\limits_0^t \Gamma(s) ds\right)\Big\{ 
2\mu\| a-\zeta(0)\|^2_{V}+\|\sigma_0-\theta(0)\|^2 $$ $$+
\int\limits_0^t \exp\left(\int\limits_s^0 \Gamma(\xi) d\xi\right)
\big[4\mu\big( E_{1}(\zeta,\theta)(s),u_m(s)-\zeta(s)\big)$$
$$+ 2 \big( E_{2}(\zeta,\theta)(s), \sigma_m(s)-\theta(s)\big)-4\mu\varepsilon_m (\zeta(s),u_m(s)-\zeta(s))_3 $$
$$ -2\varepsilon_m(\theta(s), \sigma_m(s)-\theta(s))_2 \big]\, ds \Big\}.
\ee 

Fix an arbitrary interval $ [0, T] $. Due to a priori estimate
\eqref{ae1}, without loss of generality (passing to a subsequence if
necessary) one may assume that there exist
limits
 $u =\lim\limits _ {m\to
\infty} ^ {} \widetilde {u} _ {m} $, which is weak-* in
$L _ {\infty} (0, \infty; V) $  and weak in $L _2(0, T; V) $, and
$ \sigma =\lim\limits _ {m\to \infty} ^ {} \widetilde {\sigma} _ {m}
$, which is weak-* in $L _ {\infty} (0, \infty; L_2) $ and weak in $L _2(0, T; L_2) $.

Moreover, by \eqref{ae3}, without loss of generality one may assume
that $\widetilde {u}_m '\to u '$ in $L_2(0,T;V_3^*)$, $\widetilde {\sigma}_m ' \to \sigma '
$ in $L_2(0,T;H^{-2}_N)$. This gives that $u\in C([0,T];
V_3^*)$, $\sigma\in C([0,T];
H^{-2}_N)$, and, by a well-known Lions-Magenes lemma, see e.g. \cite[Lemma 2.2.6]{book}, $u\in C_w([0,T];
V),$ $\sigma\in C_w([0,T]; L_2)$.

Take the scalar product in $L_2 (0, T) $ of inequality \eqref{in330} with
a smooth scalar function $ \psi$ with compact support in $(0,T) $
and with non-negative values, and use \eqref{ae2} and the Cauchy-Bunyakovsky-Schwarz inequality:

\be \label{in331} \int \limits_0^T \Big\{ 2\mu\|u_m(t)-\zeta(t)\|^2_{V}+
\|\sigma_m(t)-\theta(t)\|^2\Big\}\psi(t)\,dt $$ $$\leq \int
\limits_0^T
\exp\left(\int\limits_0^t \Gamma(s) ds\right)\Big\{ 
2\mu\| a-\zeta(0)\|^2_{V}+\|\sigma_0-\theta(0)\|^2 $$ $$+
\int\limits_0^t \exp\left(\int\limits_s^0 \Gamma(\xi) d\xi\right)
\big[4\mu\big( E_{1}(\zeta,\theta)(s),u_m(s)-\zeta(s)\big)$$
$$+ 2 \big( E_{2}(\zeta,\theta)(s), \sigma_m(s)-\theta(s)\big)\big]\, ds + C(\sqrt{\varepsilon_m}+\varepsilon_m) \Big\}\psi(t)\,dt.
\ee

Passing to the limit inferior as $m\to\infty $ in \eqref{in331}, and
using the fact that the norm of a weak limit of a sequence does
not exceed the limit inferior of the norms, we arrive at: \be \label{in332} \int \limits_0^T \Big\{ 2\mu\|u(t)-\zeta(t)\|^2_{V}+
\|\sigma(t)-\theta(t)\|^2\Big\}\psi(t)\,dt $$ $$\leq \int
\limits_0^T
\exp\left(\int\limits_0^t \Gamma(s) ds\right)\Big\{ 
2\mu\| a-\zeta(0)\|^2_{V}+\|\sigma_0-\theta(0)\|^2 $$ $$+
\int\limits_0^t \exp\left(\int\limits_s^0 \Gamma(\xi) d\xi\right)
\big[4\mu\big( E_{1}(\zeta,\theta)(s),u(s)-\zeta(s)\big)$$
$$+ 2 \big( E_{2}(\zeta,\theta)(s), \sigma(s)-\theta(s)\big)\big]\, ds  \Big\}\psi(t)\,dt.
\ee  Since $ \psi $ and $T$ were chosen arbitrarily, \eqref{in332} yields
\eqref{ds}, and we have proven the existence of a dissipative solution. Note that \eqref{ds} holds at every non-negative moment of time owing to the weak continuity of $u$ and $\sigma$.

Now, let $(u,\sigma)$ be a
dissipative solution with the same initial data as the
strong solution $(u_T,\sigma_T)$. Putting $\zeta=u_T$,
$\theta=\sigma_T$ in \eqref{ds} for $t\in[0,T]$, and taking into account
that $E_1(u_T,\sigma_T)\equiv E_2(u_T,\sigma_T) \equiv 0$ on
$[0,T]$, we get that the right-hand side of \eqref{ds} vanishes there,
and we arrive at the claim b) of Theorem \ref{mai}. Observe that c) is a direct consequence of a)
and b): any strong solution should coincide with all
dissipative solutions as long as it exists (cf. \cite{diss1}). \end{proof}

\begin{remark} Remember that, in Lemmas \ref{diin} and \ref{leae}, $\Omega$ can be an unbounded domain with the cone property. Moreover, the constants in \eqref{ae1} and \eqref{ae2} do not depend on $\Omega$. Then, in order to establish existence of a dissipative solution to \eqref{eq1} -- \eqref{eq6}, one can try to approximate $\Omega$ with regular bounded domains (cf. e.g. \cite{vdec}) and pass to the limit, without generalizing Lemma \ref{lews} onto the case of unbounded ones. However, some pitfalls appear, in particular, we did not manage to prove the weak continuity of the limiting function. So we leave the general ``unbounded" case as an \emph{open problem}. In the particular case $\Omega=\R^n$, Lemma \ref{lews} and Theorem \ref{mai} are valid since one can replace the operators $A_3$ and $A_2$ and the spaces $V_3$ and $H^2$ with $(I-\Delta)^3$, $(I-\Delta)^2$, the closure of $\mathcal{V}$ in $H^3$, and $H^2_0$, respectively, and approximate $\Omega$ with an ascending sequence of concentric balls (see a similar reasoning in \cite{diss1}).  
\end{remark}

\begin{remark} \label{maxw} Assume again $a\in V,$ $\sigma_0\in L_2$. Take a sequence of positive numbers $\alpha_m \to 0$. The dissipative bound \eqref{din1} guarantees compactness of any sequence $\{(u_m,\sigma_m)\}$ of dissipative solutions to the Maxwell-$\alpha_m$ problem in the weak-* topology of $L_\infty(0,\infty; H\times L_2)$. The accumulation points of $\{(u_m,\sigma_m)\}$ can be considered as ``ultra-generalized" solutions to the IBVP for motion of the Maxwell fluid, i.e. the problem \eqref{eq1} -- \eqref{eq6} with $\alpha=0$. However, the relevance of this notion is an open question.  \end{remark}

\begin{remark} A related open problem is whether our dissipative solutions to the Euler-$\alpha$ problem converge in some sense to the dissipative solutions to the Euler problem as $\alpha\to 0$. The convergence is known for the strong solutions on the time interval where they exist \cite{shiz}. \end{remark}

\begin{remark} \label{wild} \emph{Wild} solutions \cite{lel} to the Euler equations are an important example of dissipative solutions. These solutions are known to exist only for rough initial data and to be non-unique. A slightly different sort of wild solutions, which are not dissipative solutions, but exist for all initial data, was considered in \cite{wied1}. It is interesting whether the $\alpha$-models admit wild solutions. \end{remark}
 
 \section{Appendix. Dissipative solutions to Cauchy problems in Hilbert spaces}
 
 In this appendix we illustrate the idea of dissipative solution by considering (a little bit informally) 
 this notion for an abstract differential equation in a Hilbert space. We point out that this is not a universal definition of dissipative solution but just a guideline, which shows the essence of this concept. Slight natural modifications of this approach (if necessary) may lead to definitions of dissipative solutions for particular PDEs, e.g. the corresponding definitions from \cite{blions,diss1} or Definition \ref{maindef} of the current paper. 
 
 Let $X$ be a Hilbert space with inner product $(\cdot,\cdot)$ and Euclidean norm $\|\cdot\|$, and $F:[0,+\infty)\times X\to X$ be a nonlinear operator, which can be discontinuous. Assume that $F$ satisfies the following condition: \be\label{cond}(F(t,x)-F(t,y),x-y)\leq d(t,y)\|x-y\|^2,\ \forall x,y\in X, t\geq 0,\ee with some locally bounded function $d(t,y)$ of  $t\geq 0$ and $y\in X$. 
 
 \begin{remark} \label{condf} For example, \eqref{cond} holds provided  \be(F(t,x),x)\leq 0,\ \forall x\in X,\ t\geq 0,\ee and the operator $F$ can be decomposed into a sum of a linear (in $x$) operator $A$ and a quadratic (with respect to $x$ for fixed $t$) operator subject to a slight boundedness assumption. Indeed, let $$F(t,x)=A(t,x)+f(t,x,x)$$ be such a decomposition, where $f$  is a bilinear operator with respect to $(x,y)$ satisfying the bound $$\|f(t,x,y)\|\leq c(t)\|x\|\|y\|$$ with a locally bounded function $c$. Take any $x,y\in X$, and let $z=x-y$. Then $$(F(t,x)-F(t,y),z)=(F(t,z),z)+(f(t,y,z),z)+(f(t,z,y),z)$$ $$\leq 2 c(t)\|y\|\|z\|^2.$$ \end{remark}
 
 We consider the abstract Cauchy problem \be\label{ode} u'(t)=F(t,u(t)), \ t>0, \ee \be\label{iq} u(0)=a\in X.\ee
Let $\mathcal{R}$ be some fixed set of differentiable functions from $[0,+\infty)$ to $X$. The solutions $u\in \mathcal{R}$ to \eqref{ode},\eqref{iq} will be called \emph{(sufficiently) regular}. More generally, a solution $u:[0,T]\to X$, $T>0$, is said to be regular if it coincides with the restriction of a function from $\mathcal{R}$ to this interval  $[0,T]$. We assume that $\mathcal{R}$ is sufficiently large, so that the set $\{u(0)|u\in \mathcal{R}\}$ is dense in $X$. 

Take a solution $u\in\mathcal{R}$ and any test function $v\in\mathcal{R}$, and denote $$E(t,v(t))=-v'(t)+F(t,v(t)).$$ Then  \eqref{ode} implies  \be\label{ode1}(u-v)'=F(t,u)-F(t,v)+E(t,v).\ee Hence,  \be\label{ode2} (\|u-v\|^2)'=2(u'-v',u-v)=2(F(t,u)-F(t,v),u-v)+2(E(t,v),u-v)$$ $$\leq 2 d(t,v(t))\|u-v\|^2+2 (E(t,v),u-v).\ee 
 
An application of Lemma \ref{ineq} yields

\be \label{di} \|u(t)-v(t)\|^2 \leq \exp\left(\int\limits_0^t 2 d(s,v(s)) ds\right)$$ $$\times\left[\|a-v(0)\|^2+2
\int\limits_0^t \exp\left(\int\limits_s^0 2 d(\xi,v(\xi)) d\xi\right) (E(s,v(s)),u(s)-v(s))\, ds
\right].\ee

\begin{defn} \label{dsol} A weakly continuous
function $u:[0,+\infty)\to X$ is called a {\it dissipative} solution
to problem \eqref{ode},\eqref{iq} provided \eqref{di} holds for all $v \in \mathcal{R}$ and $t\geq 0$. \end{defn}

The basic properties of these solutions are given by

\begin{prop} \label{prop41}
a) If, for some $a\in X$, there exist $T>0$
and a regular solution $u_T: [0,T]\to X$ to problem \eqref{ode},\eqref{iq}, then the restriction
of any dissipative solution (with the same initial data) to
$[0,T]$ coincides with $u_T$. b) Every regular solution $u\in \mathcal{R}$ is a (unique) dissipative solution. c) Any dissipative solution satisfies the initial condition \eqref{iq}.
\end{prop}

\begin{proof} Let $u$ be a
dissipative solution with the same initial data as $u_T$. There exists $\tilde{u}_T\in \mathcal{R}$ coinciding with $u_T$ on $[0,T]$. Putting $v=\tilde{u}_T$
in \eqref{di} for $t\in[0,T]$, and taking into account
that $E(t,\tilde{u}_T(t))\equiv 0$ on
$[0,T]$, we conclude that the right-hand side of \eqref{di} vanishes there. Thus we have a). The reasoning above which led us to the derivation of inequality \eqref{di} gives us b); note that the uniqueness in b) follows from a). To show c), we put $t=0$ in \eqref{di} and get \be\|u(0)-v(0)\| \leq \|a-v(0)\|.\ee Since the set of possible $v(0)$ is dense in $X$, we have \be\label{di0}\|u(0)-b\| \leq \|a-b\|\ \forall b\in X,\ee and it remains to let $b=a$. \end{proof} 

\begin{remark} The uniqueness of dissipative solutions is not given a priori but follows from the existence of a regular solution. On the other hand, when a regular solution does not exist, non-uniqueness is possible (see Remark \ref{wild}).  
\end{remark}

\begin{remark} Consider the ``quasihomogeneous" case, i.e. when \be F(\cdot,0)\equiv 0.\ee If
\be d(\cdot,0)\equiv 0,\ee then \eqref{di} with $v\equiv 0$ yields dissipative behaviour of solutions,  
\be \label{dib} \|u(t)\| \leq \|u(0)\|,\ t> 0.\ee All this holds true, in particular, in the framework of Remark \ref{condf}. A similar situation happens for the homogeneous 3D Euler equation \cite{tit}. 
However, Definition \ref{dsol} may be useful without assumptions of this kind (cf. \cite{diss1}). 
\end{remark}

A remarkable thing is that these dissipative solutions \emph{exist} under minimal assumptions on $F$. We are not going to formulate an existence theorem, but we'll try to explain why they exist. 

The first point is an a priori estimate. The Cauchy-Bunyakovsky-Schwarz inequality and \eqref{ode2} with $v\equiv 0$ yield
\be\label{ode3} (\|u\|^2)' \leq 2 \left(d(t,0)+\frac 14\right)\|u\|^2+2\|F(t,0)\|^2,\ee 
so, by Lemma \ref{ineq}, we have the following bound:

\be \label{apr} \|u(t)\|^2 \leq \exp\left(\int\limits_0^t 2 \left(d(s,0) +\frac 14\right)\, ds\right)$$ $$\times\left[\|a\|^2+2
\int\limits_0^t \exp\left(\int\limits_s^0 2 \left(d(\xi,0) +\frac 14\right) d\xi\right) \|F(s,0)\|^2\, ds
\right].\ee

Now approximate $F$ by smooth functions $F_\epsilon$, $\epsilon > 0$, $F_\epsilon(\cdot,v(\cdot)) \to F(\cdot,v(\cdot)) $ in $L_{1,loc}(0,\infty)$ for any fixed $v\in \mathcal{R}$ as $\epsilon \to 0$. The functions $F_\epsilon$ satisfy \eqref{cond} with some $d_\epsilon$ (since they are smooth), and thus the solutions $u_\epsilon$ of the corresponding systems \eqref{ode},\eqref{iq} are a priori bounded by analogues of \eqref{apr}. The infinite-dimensional Picard-Lindel\"{o}f theorem (see e.g. \cite{reiss,shk}) implies that the  solutions $u_\epsilon$ exist (globally) and are unique. 
Assume that there exists a function $\tilde{d}(t,y)$ for which  \eqref{cond} is valid for any $\epsilon >0$.  Without loss of generality, $d\equiv \tilde{d}$ (if not, one can replace both of them by $\max[d,\tilde{d}]$ ). If the functions $F_\epsilon(\cdot,0) $ are locally square integrable in time, and the corresponding integrals are uniformly (i.e. independently of $\epsilon$) bounded, then the uniform a priori estimate \eqref{apr} holds for $u_\epsilon$. Thus, any sequence $\{u_{\epsilon_k}\}$, $\epsilon_k \to 0$, is relatively compact in the weak-* topology of $L_\infty(0,T;X)$ for any $T>0$, so without loss of generality there exists a limit $u$. A diagonal argument guarantees that the sequence and the limit can be chosen so that they do not depend on $T$.

Suppose that there are a ``large" Banach space $Y$ (so that $X$ is continuously embedded into it), 
and, say, a continuous function $c:\R^2 \to \R$ such that $$\|F_\epsilon(t,x)\|_Y \leq c(t,\|x\|_X),\ x\in X, \ t\geq 0.$$
This assumption implies a uniform bound for the time derivatives $$\|u'_\epsilon\|_{L_{\infty}(0,T;Y)}\leq C.$$
Thus, $u' \in L_{\infty}(0,T;Y)$, and, by \cite[Lemma 2.2.6]{book}, $u \in C_w([0,T];X)$. 

Now, remember that $u_{\epsilon_k}$ are, in particular, dissipative solutions. Multiplying by a smooth nonnegative function $\psi(t)$ and integrating from $0$ to $T$, we can pass to the limit in \eqref{di} in the same way as it was done in the proof of Theorem \ref{mai}, and we conclude that $u$ is a dissipative solution to the original system \eqref{ode},\eqref{iq}.

\begin{remark} Let $X$ be finite-dimensional. Then one can study system \eqref{ode},\eqref{iq} via Filippov's approach \cite{fil1,fil2}. Condition \eqref{cond} locally implies the so-called right Lipschitz condition  $$(F(t,x)-F(t,y),x-y)\leq C\|x-y\|^2,$$ which guarantees uniqueness of Filippov solutions. On the other hand, it is not clear whether the Filippov solution is always a dissipative one, or vice versa (or at least whether dissipative solutions are unique in this case). 
\end{remark}

\begin{remark} The model considered recently in \cite{v14} does not satisfy any condition similar to \eqref{cond}, but we found an analogue of \eqref{di}. However, that inequality has an absolute value in the last integral, so the weak passage to the limit is not an option. We proceed there via some strong convergence, which is still not strong enough to get classical weak solutions.\end{remark}
 
\bibliography{maxarx1}

\bibliographystyle{abbrv}
\end{document}